\documentclass[12pt]{amsart}
\usepackage{fullpage}
\usepackage{amsfonts,amscd}
\usepackage{amssymb}
\usepackage{url}
\usepackage{graphicx}
\usepackage[english]{babel}
\theoremstyle{plain}
\newtheorem{theorem}                {Theorem}      [section]
\newtheorem{proposition}  [theorem]  {Proposition}
\newtheorem{corollary}    [theorem]  {Corollary}
\newtheorem{lemma}        [theorem]  {Lemma}

\theoremstyle{definition}
\newtheorem{example}      [theorem]  {Example}
\newtheorem{remark}       [theorem]  {Remark}
\newtheorem{definition}   [theorem]  {Definition}

\numberwithin{equation}{section}

\def \R{{\mathbb R}}
\def \s{{\mathbb S}}

\def \n{{\mathbb N}}

\usepackage{color}

\def \link {~}
\def \1 {\`}

\DeclareMathOperator{\grad}{grad}
\DeclareMathOperator{\trace}{trace}
\numberwithin{equation}{section}

\title[Polyharmonic hypersurfaces into space forms]{Polyharmonic hypersurfaces into space forms}

\author{S.~Montaldo}
\address{Universit\`a degli Studi di Cagliari\\
Dipartimento di Matematica e Informatica\\
Via Ospedale 72\\
09124 Cagliari, Italia}
\email{montaldo@unica.it}

\author{C.~Oniciuc}
\address{Faculty of Mathematics\\ ``Al.I. Cuza'' University of Iasi\\
Bd. Carol I no. 11 \\
700506 Iasi, ROMANIA}
\email{oniciucc@uaic.ro}

\author{A.~Ratto}
\address{Universit\`a degli Studi di Cagliari\\
Dipartimento di Matematica e Informatica\\
Via Ospedale 72\\
09124 Cagliari, Italia}
\email{rattoa@unica.it}

\usepackage{hyperref}
\begin{document}
\begin{abstract}
In this paper we shall assume that the ambient manifold is a space form $N^{m+1}(c)$ and we shall consider polyharmonic hypersurfaces of order $r$ (briefly, $r$-harmonic), where $r\geq 3$ is an integer. For this class of hypersurfaces we shall prove that, if $c \leq 0$, then any $r$-harmonic hypersurface must be minimal provided that the mean curvature function and the squared norm of the shape operator are constant. When the ambient space is $\s^{m+1}$, we shall obtain the geometric condition which characterizes the $r$-harmonic hypersurfaces with constant mean curvature and constant squared norm of the shape operator, and we shall establish the bounds for these two constants.  In particular, we shall prove the existence of several new examples of proper $r$-harmonic isoparametric hypersurfaces in $\s^{m+1}$ for suitable values of $m$ and $r$. Finally, we shall show that all these $r$-harmonic hypersurfaces are also $ES-r$-harmonic, i.e., critical points of the Eells-Sampson $r$-energy functional.  
\end{abstract}

\subjclass[2000]{Primary: 58E20; Secondary: 53C43.}

\keywords{$r$-harmonic maps, the Chen-Maeta conjecture, submanifolds of the sphere, isoparametric hypersurfaces}

\thanks{The authors S.M. and A.R. were supported by Fondazione di Sardegna (project STAGE) and Regione Autonoma della Sardegna (Project KASBA); the author C.O. was supported by a project funded by the Ministry of
Research and Innovation within Program 1 - Development of the national RD system, Subprogram 1.2 - Institutional Performance - RDI excellence funding projects, Contract no. 34PFE/19.10.2018.}

\maketitle

\section{Introduction and statement of the results}\label{Intro}
Let us consider a submanifold of the Euclidean space described by the isometric immersion $\varphi: M^m \hookrightarrow \R^{n}$. We say that $M^m$ is polyharmonic of order $r$ (briefly, $r$-harmonic) if 
\begin{equation}\label{r-harmonicity-in-R-m+1}
\Delta^r \varphi=\left (\Delta^r \varphi^1,\ldots,\Delta^r \varphi^{n}\right )=0 \,,
\end{equation} 
where $r\geq 1$ is an integer and $\Delta$ denotes the Laplace operator associated to the induced metric on $M^m$. Now, let ${\mathbf H}$ be the mean curvature vector field of $M^m$. Since
\[
{\mathbf H}= -\frac{1}{m} \Delta\varphi\,,
\]
equation \eqref{r-harmonicity-in-R-m+1} is equivalent to
\begin{equation*}\label{r-harmonicity-in-R-m+1-con-H}
\Delta^{r-1} {\mathbf H}=0 \,.
\end{equation*}
In particular, \textit{minimal} submanifolds are trivially $r$-harmonic for all $r \geq 1$. Therefore, when $r\geq2$, we say that an $r$-harmonic submanifold is \textit{proper} if it is \textit{not} minimal. The case $r=2$ corresponds to the so-called biharmonic submanifolds and it is the most studied in the literature (for instance, see \cite{Chen, Jiang, SMCO, Ou}). In particular, the Chen conjecture states that any biharmonic submanifold of $\R^{n}$ is minimal. Although there are some results which prove that the conjecture holds under suitable geometric restrictions (see \cite{Chen, Chen2, HasVla95, AMPA, Geom-Phys}), the general case is still open. In a similar spirit, the Maeta conjecture (see \cite{Maeta1, Maeta4, Na-Ura}) states that any $r$-harmonic submanifold of $\R^{n}$ is minimal. Also in this case some partial results are available. For instance, the Maeta conjecture is true for curves (see \cite{Maeta1}), but again the general case is open. 

For the purposes of this paper, it is necessary to extend these notions to the case that the ambient space is a general manifold $N$. To this end, the starting point is the notion of a \textit{harmonic map}. 
Harmonic maps are the critical points of the {\em energy functional}
\begin{equation}\label{energia}
E(\varphi)=\frac{1}{2}\int_{M}\,|d\varphi|^2\,dV \, ,
\end{equation}
where $\varphi:M\to N$ is a smooth map between two Riemannian
manifolds $(M,g)$ and $(N,h)$. In particular, $\varphi$ is harmonic if it is a solution of the Euler-Lagrange system of equations associated to \eqref{energia}, i.e.,
\begin{equation}\label{harmonicityequation}
  - d^* d \varphi =   {\trace} \, \nabla d \varphi =0 \, .
\end{equation}
The left member of \eqref{harmonicityequation} is a vector field along the map $\varphi$ or, equivalently, a section of the pull-back bundle $\varphi^{-1} TN$: it is called {\em tension field} and denoted $\tau (\varphi)$. In addition, we recall that, if $\varphi$ is an \textit{isometric immersion}, then $\varphi$ is a harmonic map if and only if the immersion $\varphi$ defines a minimal submanifold of $N$ (see \cite{EL83, EL1} for background). Let us denote $\nabla^M$, $\nabla^N$ and $\nabla^{\varphi}$ the induced connections on the bundles $TM$, $TN$ and $\varphi ^{-1}TN$ respectively. The \textit{rough Laplacian} on sections of $\varphi^{-1}  TN$, denoted $\overline{\Delta}$, is defined by
\begin{equation} \label{roughlaplacian}
    \overline{\Delta}=d^* d =-\sum_{i=1}^m\left(\nabla^{\varphi}_{e_i}
    \nabla^{\varphi}_{e_i}-\nabla^{\varphi}_
    {\nabla^M_{e_i}e_i}\right)\,,
\end{equation}
where $\{e_i\}_{i=1}^m$ is a local orthonormal frame field tangent to $M$. 

Now, in order to define the notion of an $r$-harmonic map, we consider the following family of functionals, which represent a version of order $r$ of the classical energy \eqref{energia}. If $r=2s$, $s \geq 1$:
\begin{eqnarray}\label{2s-energia}
E_{2s}(\varphi)&=& \frac{1}{2} \int_M \, \langle \, \underbrace{(d^* d) \ldots (d^* d)}_{s\, {\rm times}}\varphi, \,\underbrace{(d^* d) \ldots (d^* d)}_{s\, {\rm times}}\varphi \, \rangle_{_N}\, \,dV \nonumber\\ 
&=& \frac{1}{2} \int_M \, \langle \,\overline{\Delta}^{s-1}\tau(\varphi), \,\overline{\Delta}^{s-1}\tau(\varphi)\,\rangle_{_N} \, \,dV\,.
\end{eqnarray}
In the case that $r=2s+1$:
\begin{eqnarray}\label{2s+1-energia}
E_{2s+1}(\varphi)&=& \frac{1}{2} \int_M \, \langle\,d\underbrace{(d^* d) \ldots (d^* d)}_{s\, {\rm times}}\varphi, \,d\underbrace{(d^* d) \ldots (d^* d)}_{s\, {\rm times}}\varphi\,\rangle_{_N}\, \,dV\nonumber \\
&=& \frac{1}{2} \int_M \,\sum_{j=1}^m \langle\,\nabla^\varphi_{e_j}\, \overline{\Delta}^{s-1}\tau(\varphi), \,\nabla^\varphi_{e_j}\,\overline{\Delta}^{s-1}\tau(\varphi)\, \rangle_{_N} \, \,dV \,.
\end{eqnarray}
We say that a map $\varphi$ is \textit{$r$-harmonic} if, for all variations $\varphi_t$,
$$
\left .\frac{d}{dt} \, E_{r}(\varphi_t) \, \right |_{t=0}\,=\,0 \,\,.
$$
In the case that $r=2$, the functional \eqref{2s-energia} is called \textit{bienergy} and its critical points are the so-called \textit{biharmonic maps}. At present, a very ample literature on biharmonic maps is available and, again, we refer to \cite{Chen, Jiang, SMCO, Ou} for an introduction to this topic. More generally, the \textit{$r$-energy functionals} $E_r(\varphi)$ defined in \eqref{2s-energia}, \eqref{2s+1-energia} have been intensively studied (see \cite{Volker, Volker2, Maeta1, Maeta3, Maeta2, Mont-Ratto4, Na-Ura, Wang, Wang2}, for instance). In particular, the Euler-Lagrange equations for $E_r(\varphi)$ were obtained by Wang \cite{Wang} and Maeta \cite{Maeta1}. We say that an $r$-harmonic map is {\it proper} if it is not harmonic (similarly, an $r$-harmonic submanifold, i.e., an $r$-harmonic isometric immersion, is {\it proper} if it is not minimal). As a general fact, when the ambient space has nonpositive sectional curvature there are several results which assert that, under suitable conditions, an $r$-harmonic submanifold is minimal (see \cite{Chen}, \cite{Maeta1}, \cite{Maeta4} and \cite{Na-Ura}, for instance).

Things drastically change when the ambient space is positively curved. Let us denote by $\s^{m+1}(R)$ the Euclidean sphere of radius $R$ and write $\s^{m+1}$ for $\s^{m+1}(1)$. Moreover, let us indicate by $A$ the shape operator of $M^m$ into $\s^{m+1}$ and by ${\mathbf H}=f \eta$ the mean curvature vector field, where $\eta$ is the unit normal vector field and $f$ is the mean curvature function. Throughout the whole paper, when we write that $M^m$ is a CMC hypersurface we mean that $f$ is a constant which will be denoted by $\alpha$. 

The first fundamental result in the theory of biharmonic hypersurfaces of $\s^{m+1}$ can be stated as follows:
\begin{theorem}\label{Th-bihar-hypersurfaces-spheres}
\cite{CMO02, CMO03, Jiang} Let $M^m$ be a non-minimal CMC hypersurface in $\s^{m+1}$. Then $M^m$ is proper biharmonic if and only if $|A|^2=m$.
\end{theorem}
In a similar spirit, in the triharmonic case Maeta obtained a result which implies: 
\begin{theorem}\label{Th-trihar-hypersurfaces-spheres}
\cite{Maeta2} Let $M^m$ be a non-minimal hypersurface in $\s^{m+1}$ and assume that $\nabla A=0$. Then $M^m$ is proper triharmonic if and only if $|A|^4-m|A|^2-m^2\alpha^2=0$.
\end{theorem}
The geometric restrictions given in Theorems\link\ref{Th-bihar-hypersurfaces-spheres} and \ref{Th-trihar-hypersurfaces-spheres} are of great interest in the general theory of submanifolds of the sphere. For instance, we know that the small hypersphere $\s^m(1/\sqrt 2)$ and certain generalised Clifford tori verify the conditions of Theorem\link\ref{Th-bihar-hypersurfaces-spheres}. Similarly, $\s^m(1/\sqrt 3)$ and some suitable generalised Clifford tori are triharmonic, but a complete classification of proper biharmonic (triharmonic) hypersurfaces in $\s^{m+1}$ is a challenging open problem, even in the compact CMC case. For the sake of completeness, we mention that, when $r \geq 4$, using completely different methods (equivariant differential geometry and variational principles), we recently proved that the small hypersphere $\s^m(1/\sqrt r)$ and certain suitable generalised Clifford tori are $r$-harmonic (see \cite{BMOR1, Mont-Ratto4}). 

We are now in the right position to state the main results of this paper. We shall write $N^{m+1}(c)$ to indicate an $(m+1)$-dimensional space form of curvature $c$ ($m \geq1, c\in \R$).
\begin{theorem}\label{Th-non-existence-surfaces-c<=0}
Let $M^2$ be a CMC triharmonic surface in $N^3(c)$, $c \leq0$. Then $M^2$ is minimal.
\end{theorem} 
\begin{theorem}\label{Th-non-existence-hypersurfaces-c<=0}
Let $M^m$ be a CMC $r$-harmonic hypersurface in $N^{m+1}(c)$, where $r\geq 3$ and $c \leq0$. If $|A|^2$ is constant, then $M^m$ is minimal.
\end{theorem} 
\begin{remark} The conclusion of Theorem\link\ref{Th-non-existence-hypersurfaces-c<=0} was obtained by Maeta (see \cite{Maeta1}) in the case that $r=3$ and $M^m$ is compact.
\end{remark}
Next, we focus on the case that the ambient space is the Euclidean sphere. Our results are:
\begin{theorem}\label{Th-structure-surfaces-c>0}
Let $M^2$ be a CMC proper triharmonic surface in $\s^3$. Then $M^2$ is an open part of the small hypersphere $\s^2(1/\sqrt 3)$.
\end{theorem} 
\begin{theorem}\label{Th-compact}
Let $M^m$ be a compact, CMC proper triharmonic hypersurface in $\s^{m+1}$. Then either $|A|^2=2m $ and $M^m=\s^m\left( 1/\sqrt 3 \right )$, or there exists a point $p\in M^m$ such that $m<|A|^2(p)<2m$.
\end{theorem}
The following is an immediate consequence of Theorem\link\ref{Th-compact}:
\begin{corollary}\label{Cor-aggiunto} Let $M^m$ be a compact, CMC proper triharmonic hypersurface in $\s^{m+1}$. If $|A|^2 \geq 2m $, then $|A|^2 = 2m $ and $M^m=\s^m\left( 1/\sqrt 3 \right )$.
\end{corollary}
\begin{theorem}\label{Th-structure-hypersurfaces-c>0}
Let $M^m$ be a CMC proper $r$-harmonic  hypersurface in $\s^{m+1}$, $r \geq3$, and assume that $|A|^2$ is constant. Then
\begin{itemize}
\item[\rm (i)] $|A|^2 \in (m,m(r-1)]$, and $|A|^2=m(r-1)$ if and only if $M^m$ is an open part of $\s^m(1/\sqrt r)$.
\item[\rm (ii)] $\alpha^2 \in (0,r-1]$, and $\alpha^2=r-1$ if and only if $M^m$ is an open part of $\s^m(1/\sqrt r)$.
\end{itemize}
\end{theorem} 
Theorem\link\ref{Th-structure-hypersurfaces-c>0} describes the structure of a certain family of proper $r$-harmonic hypersurfaces of the Euclidean sphere. The fact that this family contains several significant examples will be illustrated by means of the following result:
\begin{theorem}\label{Th-existence-hypersurfaces-c>0}
Let $M^m$ be a non-minimal CMC hypersurface in $\s^{m+1}$ and assume that $|A|^2$ is constant. Then
$M^m$ is proper $r$-harmonic ($r \geq 3$) if and only if
\begin{equation}\label{r-harmonicity-condition-in-spheres}
|A|^4-m|A|^2-(r-2)m^2 \alpha^2=0 \,.
\end{equation}
\end{theorem} 
\begin{remark} The special case $r=3$ in Theorem~\ref{Th-existence-hypersurfaces-c>0} is a slight improvement of Theorem\link\ref{Th-trihar-hypersurfaces-spheres}. We also observe that the conclusions of Theorems\link\ref{Th-non-existence-hypersurfaces-c<=0} and \ref{Th-existence-hypersurfaces-c>0} remain true for $r=2$ even without the hypothesis $|A|^2$ equal to a constant.
\end{remark} 
In order to illustrate the main applications of Theorem\link\ref{Th-existence-hypersurfaces-c>0}, it is now necessary to recall some basic facts about \textit{isoparametric hypersurfaces} of the Euclidean sphere. These hypersurfaces have a long, beautiful history which started in 1918 with the work of Laura and continued with the contributions of several authors, amongst which we cite E. Cartan, Abresch, Ozeki-Takeuchi, M\"{u}nzner, Stolz, Chi and, more recently, Miyaoka (we refer to \cite{Abresch, Baird-book, Chi, Miyaoka, Miyaoka2, Munzner, Munzner2, Nomizu, Oz-Tak1, Oz-Tak2, Stolz} and references therein). A smooth function $F:\s^{m+1} \to \R$ is called \textit{isoparametric} if both $|\nabla F|^2$ and $\Delta F$ are functions of $F$. Isoparametric hypersurfaces are the regular level sets of isoparametric  functions. Each isoparametric hypersurface is CMC and has $\ell$ distinct constant principal curvatures $k_1,\ldots,k_\ell$ with constant multiplicities $m_1,\ldots,m_\ell$, $m_1+\ldots +m_\ell=m$. Moreover, the only possible values for $\ell$ are $1,2,3,4,6$, and $m_{i+2}=m_i$, so that there are at most two distinct multiplicities, which we shall denote $m_1,m_2$. The integer $\ell$ is called the \textit{degree} of the isoparametric hypersurface. Any isoparametric function $F$ has image $F(\s^{m+1})=[-1,1]$. In order to state our results, it is convenient to describe a family of parallel isoparametric hypersurfaces as follows. As in \cite{Baird-book}, we define $M_s=F^{-1}(\cos \ell s)$, where $0<s< \pi/\ell$. Geometrically, $s$ represents the affine parameter such that $\nabla F / |\nabla F| = \nabla s$, and so $\nabla s$ is normal to $M_s$.
In the cases $\ell=1,2$ the isoparametric hypersurfaces $M_s$ are small hyperspheres and generalised Clifford tori respectively. These cases have been thoroughly studied, by different methods, in \cite{BMOR1, Maeta2, Mont-Ratto4, Mont-Ratto5}. Therefore, we focus on the remaining cases $\ell=3,4,6$. By a suitable choice of the orientation we can assume $m_1 \leq m_2$ and we have the following explicit description of the principal curvatures of $M_s$, $s \in (0,\pi/\ell)$:
\begin{equation}\label{principal-curvatures}
k_i(s)= \cot \left( s +\frac{(i-1)\pi}{\ell} \right ), \quad i=1,\ldots,\ell\,.
\end{equation} 
For a fixed $s$, the distributions on $M_s$ corresponding to the $k_i$'s are integrable and the leaves are totally umbilical in $\s^{m+1}$, i.e., they are small spheres of dimension $m_i$
and radius $\sin(s+(i-1) \pi/\ell)$, $i=1,\ldots,\ell$. Moreover, these leaves are totally geodesic in $M_s$.

The non-existence of proper biharmonic isoparametric hypersurfaces of degree $\ell=3,4,6$ was proved in \cite{Urakawa-no-bihar}. Our results will extend this to the cases $r=3,4$, but, on the other hand, the instances described in Example\link\ref{Ex-r=5} will enable us to conclude that there exist proper $r$-harmonic isoparametric hypersurfaces of degree $\ell=4$ for all $ r\geq 5$. Finally, we point out that the analysis of \cite{Urakawa-no-bihar} shows that, when $\ell=3,4$ or $6$, $|A|^2 >m$. This implies that, in this context, a solution of \eqref{r-harmonicity-condition-in-spheres} cannot be minimal. 
\vspace{1mm}

\textbf{Case $\ell=3$.} There exist only four examples of this type, corresponding to $m_1=m_2=1,2,4,8$. Our result is:
\begin{theorem}\label{Th-isoparametric-p=3} Let $M_s$, $0<s<\pi/3$, be a family of parallel isoparametric hypersurfaces of degree $3$ in $\s^{m+1}$. Then
\begin{itemize}
\item[{\rm (i)}] if $2 \leq r \leq 19$, the family $M_s$ does not contain any proper $r$-harmonic hypersurface;
\item[{\rm (ii)}] if $r \geq 20$, the family $M_s$ contains four (two geometrically distinct) proper $r$-harmonic hypersurfaces.
\end{itemize}
\end{theorem}

\textbf{Case $\ell=6$.} It is known that in this case necessarily $m_1=m_2$ and the possible values are $m_1=1$ or $m_1=2$. There exist precisely two \textit{homogeneous} examples of this type, with $m_1=1$ and $m_1=2$ respectively. If $m_1=1$, there exists no non-homogeneous examples. By contrast, when $m_1=2$ the existence of \textit{non-homogeneous} instances is still an open problem (see \cite{Siffert}). Our result is:
\begin{theorem}\label{Th-isoparametric-p=6} Let $M_s$, $0<s<\pi/6$, be a family of parallel isoparametric hypersurfaces of degree $6$ in $\s^{m+1}$. Then
\begin{itemize}
\item[{\rm (i)}] if $2 \leq r \leq 109$, the family $M_s$ does not contain any proper $r$-harmonic hypersurface;
\item[{\rm (ii)}] if $r \geq 110$, the family $M_s$ contains four proper $r$-harmonic hypersurfaces.
\end{itemize}
\end{theorem}
\textbf{Case $\ell=4$.} This is the richest case. The only examples with $m_1=m_2$ occur when $m_1=m_2=1$ or $m_1=m_2=2$. But we have examples with $m_1\neq m_2$ provided that $m_1=4$ and $m_2=5$, or $1+m_1+m_2$ is a multiple of $2^{\xi(m_1-1)}$, where $\xi(n)$ denotes the number of natural numbers $p$ such that $1 \leq p \leq n$ and $p \equiv 0,1,2,4 \,({\rm mod}\, 8)$. Let $C_m$ be the skew-symmetric Clifford algebra over $\R$ generated by the canonical basis $\left \{e_1,\ldots,e_m \right \}$ of $\R^m$ subject to the only constraint
\[
e_i e_j+e_j e_i = -2\delta_{ij} {\rm I}\,,
\]
where ${\rm I}$ is the identity element of $C_{m}$.
A more explicit description of the pairs $(m_1,m_2)$ which can occur is the following (see \cite{Chi}):
\[
(m_1,m_2)=(m,k \delta_m -m-1)\,,
\]
where $\delta_m$ denotes the dimension of an irreducible  module of $C_m$ and the positive integers $m,k$ are such that the second entry is positive. 
A full classification of the corresponding isoparametric hypersurfaces is available except in the case $(m_1,m_2)=(7, 8)$. Finally, for future reference, we point out that the following pairs of multiplicities occur: $m_1=1,m_2\geq 2$; $m_1=2,m_2=2k,k\geq 2$; $m_1=3,m_2=4$; $m_1=4,m_2=4k-5,k\geq 3$. Our results are:
\begin{theorem}\label{Th-isoparametric-p=4} Let $M_s$, $0<s<\pi/4$, be a family of parallel isoparametric hypersurfaces of degree $4$ in $\s^{m+1}$ and assume that $m_1=m_2$. Then
\begin{itemize}
\item[{\rm (i)}] if $2 \leq r \leq 41$, the family $M_s$ does not contain any proper $r$-harmonic hypersurface;
\item[{\rm (ii)}] if $r \geq 42$, the family $M_s$ contains four proper $r$-harmonic hypersurfaces.
\end{itemize}
\end{theorem}
\begin{theorem}\label{Th-isoparametric-p=4-m1-diverso-m2} Let $M_s$, $0<s<\pi/4$, be a family of parallel isoparametric hypersurfaces of degree $4$ in $\s^{m+1}$ and assume that $m_1< m_2$. Let $b=m_2/m_1$ and define 
\begin{eqnarray}\label{Pol-grado-4-p=4}\nonumber
P_{b,r}(y)&=&y^4 \left(b^2 r+2 b r+r\right)+y^3 \left(-3 b^2 r-4 b^2-4 b r-r+4\right)+y^2 \left(3 b^2 r+12 b^2+2 b r-2 b\right)\\
&&+y \left(-b^2 r-12 b^2+2 b\right)+4 b^2 \,.
\end{eqnarray}
Then the family $M_s$ contains a proper $r$-harmonic hypersurface if and only if the fourth order polynomial $P_{b,r}(y)$ admits a root $y^* \in (0,1)$, and to any such root corresponds a proper $r$-harmonic hypersurface $M_s$ with $s=(1/2)\arccos \left (\sqrt {y^*} \right )$. 

Moreover, there exist two natural numbers $r^{**}\geq r^* \geq 5$ such that
\begin{itemize}
\item[{\rm (i)}] if $2 \leq r < r^*$, the family $M_s$ does not contain any proper $r$-harmonic hypersurface;
\item[{\rm (ii)}] if $r \geq r^*$, the family $M_s$ contains at least one proper $r$-harmonic hypersurface;
\item[{\rm (iii)}] if $r \geq r^{**}$, the family $M_s$ contains four proper $r$-harmonic hypersurfaces.
\end{itemize}
\end{theorem}
\begin{example}\label{Ex-r=5} We observe that $P_{b,r}(0)=4 b^2$ and $P_{b,r}(1)=4$. It follows that, in most examples, there is an even (possibly zero) number of roots in $(0,1)$. For instance, performing a numerical analysis with the software Mathematica${}^{\footnotesize \textregistered} $, we were able to check that, when $m_1=1$ and $ m_2=10000$, there are at least two solutions for all $r\geq r^*= 5$, and $r^{**}=312919$. This shows that, in general, the estimate $r^* \geq 5$ in Theorem\link\ref{Th-isoparametric-p=4-m1-diverso-m2} is sharp. In subsection\link\ref{Remark-four-solutions} we shall provide some upper bound for $r^{*}$ and $r^{**}$.
\end{example}
\vspace{2mm}

\subsection{The definition of Eells and Sampson} The notion of $r$-harmonicity which we have used in this paper represents, both from the geometric and the analytic point of view, a convenient approach to the study of higher order versions of the classical energy functional. On the other hand, we point out that the first idea of studying higher order energies was formulated in a different way. More precisely, in 1965 Eells and Sampson (see \cite{ES}) proposed the following functionals which we denote $E^{ES}_r(\varphi)$ to remember these two outstanding mathematicians:
\begin{equation}\label{ES-energia}
    E_r^{ES}(\varphi)=\frac{1}{2}\int_{M}\,|(d^*+d)^r (\varphi)|^2\,dV\,\, .
\end{equation}
Again, harmonic maps are trivially absolute minima for the functionals \eqref{ES-energia}. Therefore, we say that $\varphi$ is a \textit{proper} $ES-r$-\textit{harmonic map} if it is a \textit{not} harmonic critical point of $ E_r^{ES}(\varphi)$. The study of \eqref{ES-energia} was suggested again in \cite{EL83}, but so far not much is known about this subject. 
The functionals $E^{ES}_r(\varphi)$ and $E_r(\varphi)$ coincide when $r=2,3$ and the difference between them appears when we consider the case $r=4$, where we have:
\begin{equation*}\label{ES-4-energia}
E^{ES}_4(\varphi)= \frac{1}{2} \int_M |d^2(\tau(\varphi)) |^2 \,dV +E_4(\varphi)\,.
\end{equation*}
As observed in \cite{EL83}, $d^2(\tau(\varphi)) $ is a $2$-form with values in $\varphi^{-1}TN$ which can be described as follows:
\begin{equation}\label{d2-description}
d^2(\tau(\varphi))(X,Y)=R^N \big(d\varphi(X),d\varphi(Y)\big )\tau(\varphi) \,.
\end{equation}
In general, aside from the special cases where $\dim M=1$ or the target $N$ is flat, the curvature term $d^2(\tau(\varphi))$ does not vanish and gives rise to difficulties which increase as $r$ does.
In particular, by contrast to the case of $E_r(\varphi)$, the explicit derivation of the Euler-Lagrange equation for the Eells-Sampson functionals $E^{ES}_r(\varphi)$ seems, in general, a very complicated task. These problems are explained in detail in the recent paper \cite{BMOR1}, where the Euler-Lagrange equation of the functional $E^{ES}_4(\varphi)$ was computed. 

Nevertheless, in the context of the present paper, we can establish a significant relationship between $r$-harmonicity and $ES-r$-harmonicity. That is achieved proving the following rather general result, where, in a standard way, $\overline{\Delta}^0$ denotes the identity operator:
\begin{theorem}\label{Th-Es-r-e-r-harmonicity-general}
Let $\varphi:(M^m,g) \to  N^{n}(c)$ be any smooth map from a Riemannian manifold $(M^m,g)$ into a space form. Let  $r\geq 4$ and assume that
\begin{equation}\label{Ipotesi-Deltatau-ortogonale}
\overline{\Delta}^i \tau(\varphi) (x) \perp d\varphi(T_x M) \quad \forall x \in M, \quad i=0, \ldots, r-4\,. 
\end{equation}
Then $\varphi$ is $r$-harmonic if and only if it is $ES-r$-harmonic.
\end{theorem}
The following corollary of the previous theorem is very important in our context:
\begin{corollary}\label{Cor-Es-r-e-r-harmonicity} Let $\varphi:M^m \to  N^{m+1}(c)$ be a CMC hypersurface with $|A|^2$ equal to a constant. Then $\varphi$ is $r$-harmonic if and only if it is $ES-r$-harmonic.
\end{corollary}
We point out that, as a consequence of Corollary\link\ref{Cor-Es-r-e-r-harmonicity}, the conclusion of Theorem\link\ref{Th-existence-hypersurfaces-c>0} holds if, in its statement, the word $r$-harmonic is replaced by $ES-r$-harmonic. The same occurs for our results on isoparametric hypersurfaces and, in particular, \textit{all the proper $r$-harmonic hypersurfaces obtained in Theorems\link\ref{Th-isoparametric-p=3}, \ref{Th-isoparametric-p=6}, \ref{Th-isoparametric-p=4} and \ref{Th-isoparametric-p=4-m1-diverso-m2} are also proper $ES-r$-harmonic. }
\section{Proof of the first general results}\label{Sec-proofs}
In this section we prove Theorems~\ref{Th-non-existence-surfaces-c<=0}--\ref{Th-existence-hypersurfaces-c>0}. As a preliminary step, in the following lemma we state without proof some standard facts which we shall use in this section.
\begin{lemma}\label{Lemma-tecnico-1} Let $\varphi:M^m \to  N^{m+1}(c)$ be a hypersurface. Let $A$ denote the shape operator and $f=(1/m) \trace A$ the mean curvature function. Then
\begin{itemize}
\item[{\rm (a)}] $(\nabla A) (\cdot,\cdot)$ is symmetric;
\item[{\rm (b)}] $\langle (\nabla A) (\cdot,\cdot), \cdot \rangle$ is totally symmetric;
\item[{\rm (c)}] $\trace (\nabla A) (\cdot,\cdot)= m \grad f$. 
\end{itemize}
\end{lemma}
Next, we perform our first computation:
\begin{lemma}\label{Lemma-Delta-H} Let $\varphi:M^m \to  N^{m+1}(c)$ be a hypersurface and denote by $\eta$ the unit normal vector field. Then
\begin{equation}\label{Delta-H-formula}
 \overline{\Delta} {\mathbf H}= (\Delta f + f |A|^2) \eta + 2 A(\grad f ) + m f \grad f \,.
\end{equation}
\end{lemma}
\begin{proof} We work with a geodesic frame field $\left \{ X_i \right \}_{i=1}^m$ around an arbitrarily fixed point $p \in M^m$. Also, we simplify the notation writing $\nabla$ for $\nabla^{M}$. Since ${\mathbf H}=f \eta$, around $p$ we have:
\[
\nabla_{X_i}^{\varphi} {\mathbf H}=\nabla_{X_i}^{\perp} {\mathbf H}-A_{{\mathbf H}}(X_i)=\left ( X_i f \right )\eta -f A  \left (X_i \right )\,.
\]
Then, denoting by $B$ the second fundamental form, at $p$ we have:
\begin{eqnarray*}
\nabla_{X_i}^{\varphi}\nabla_{X_i}^{\varphi} {\mathbf H}&=&\left(X_i X_i f \right )\eta- \left( X_i f \right )A  \left (X_i \right )- \left( X_i f \right )A  \left (X_i \right )-f \big( \nabla_{X_i} A  \left (X_i \right )+B\left( X_i, A  \left (X_i \right )\right )\big )\nonumber\\
&=&\left(X_i X_i f \right )\eta- 2\left( X_i f \right )A  \left (X_i \right )-f (\nabla A)(X_i,X_i)-f|A  \left (X_i \right )|^2 \eta \,.
\label{eq:nabla2H}
\end{eqnarray*}
Now, taking the sum over $i$ and using Lemma~\ref{Lemma-tecnico-1}, we obtain \eqref{Delta-H-formula} (note that the sign convention for $\Delta$ and $\overline{\Delta}$ is as in \eqref{roughlaplacian}).
\end{proof}
Next, we assume that the mean curvature function $f$ is constant and we obtain
\begin{lemma}\label{Lemma-Delta2-H} Let $\varphi:M^m \to  N^{m+1}(c)$ be a hypersurface and assume that its mean curvature function $f$ is equal to a constant $\alpha$. Then
\begin{equation}\label{Delta-2-H-formula}
 \overline{\Delta}^2 {\mathbf H}= \alpha \left (\Delta |A|^2 + |A|^4 \right ) \eta + 2 \alpha A \big(\grad |A|^2\big )  \,.
\end{equation}
\end{lemma}
\begin{proof} Since $f$ is constant, according to Lemma\link\ref{Lemma-Delta-H} we have $\overline{\Delta}{\mathbf H}=\alpha |A|^2 \eta$ and so
\[
\overline{\Delta}^2{\mathbf H}=\alpha \overline{\Delta}\big ( |A|^2 \eta \big)\,.
\]
Now, around $p$:
\[
\nabla_{X_i}^{\varphi} \big ( |A|^2 \eta\big )=\big ( X_i |A|^2 \big )\eta -| A|^2  A\left (X_i \right )\,.
\]
At $p$:
\begin{eqnarray*}
\nabla_{X_i}^{\varphi}\nabla_{X_i}^{\varphi} \big ( |A|^2 \eta\big )&=&\left(X_i X_i |A|^2 \right )\eta- \left( X_i |A|^2 \right )A  \left (X_i \right )\\
&&- \left( X_i |A|^2 \right )A  \left (X_i \right )-|A|^2 \big( \nabla_{X_i} A  \left (X_i \right )+B\left( X_i, A  \left (X_i \right )\right )\big )\,.\\
\end{eqnarray*}
Taking the sum over $i$ and using Lemma~\ref{Lemma-tecnico-1} we find
\[
\overline{\Delta}\big ( |A|^2 \eta \big)=\left (\Delta |A|^2 \right ) \eta+ 2  A \big(\grad |A|^2\big )  +m |A|^2 \grad f + |A|^4 \eta
\]
and, since $f$ is constant, the proof ends immediately.
\end{proof}
Next, we recall the formulas which describe the $r$-tension field of a general map $\varphi:M \to N$ between two Riemannian manifolds (see \cite{Maeta1}):
\begin{eqnarray}\label{2s-tension}
\tau_{2s}(\varphi)&=&\overline{\Delta}^{2s-1}\tau(\varphi)-R^N \left(\overline{\Delta}^{2s-2} \tau(\varphi), d \varphi (e_i)\right ) d \varphi (e_i) \nonumber\\ 
&&  -\, \sum_{\ell=1}^{s-1}\, \left \{R^N \left( \nabla^\varphi_{e_i}\,\overline{\Delta}^{s+\ell-2} \tau(\varphi), \overline{\Delta}^{s-\ell-1} \tau(\varphi)\right ) d \varphi (e_i)  \right .\\ \nonumber
&& \qquad \qquad  -\, \left . R^N \left( \overline{\Delta}^{s+\ell-2} \tau(\varphi),\nabla^\varphi_{e_i}\, \overline{\Delta}^{s-\ell-1} \tau(\varphi)\right ) d \varphi (e_i)  \right \} \,\, ,
\end{eqnarray}
where $\overline{\Delta}^{-1}=0$ and $\{e_i\}_{i=1}^m$ is a local orthonormal frame field tangent to $M$ (the sum over $i$ is not written but understood). Similarly,
\begin{eqnarray}\label{2s+1-tension}
\tau_{2s+1}(\varphi)&=&\overline{\Delta}^{2s}\tau(\varphi)-R^N \left(\overline{\Delta}^{2s-1} \tau(\varphi), d \varphi (e_i)\right ) d \varphi (e_i)\nonumber \\ 
&&  -\, \sum_{\ell=1}^{s-1}\, \left \{R^N \left( \nabla^\varphi_{e_i}\,\overline{\Delta}^{s+\ell-1} \tau(\varphi), \overline{\Delta}^{s-\ell-1} \tau(\varphi)\right ) d \varphi (e_i)  \right .\\ \nonumber
&& \qquad \qquad  -\, \left . R^N \left( \overline{\Delta}^{s+\ell-1} \tau(\varphi),\nabla^\varphi_{e_i}\, \overline{\Delta}^{s-\ell-1} \tau(\varphi)\right ) d \varphi (e_i)  \right \} \\ \nonumber
&& \,-\,R^N \Big( \nabla^\varphi_{e_i}\,\overline{\Delta}^{s-1} \tau(\varphi), \overline{\Delta}^{s-1} \tau(\varphi)\Big ) d \varphi (e_i)\,\,. 
\end{eqnarray}

We shall also need the expression for the sectional curvature tensor field in the special case that $N$ is a space form:
\begin{equation}\label{tensor-curvature-N(c)}
R^{N(c)}(X,Y)Z=c\, \big(\langle Y,Z \rangle X-\langle X,Z \rangle Y \big) \quad \quad  \forall \,X,Y,Z \in C(TN(c)) \,.
\end{equation}

We are now in the right position to prove our first theorem.
\begin{proof}[Proof of Theorem\link\ref{Th-non-existence-surfaces-c<=0}] The $3$-tension field is described by \eqref{2s+1-tension} with $s=1$. In the first part of the proof, for future reference, we do not make any assumption on the dimension $m$ and the curvature $c$. We observe that $\tau(\varphi)=m{\mathbf H}$ and use Lemma\link\ref{Lemma-Delta-H}, Lemma\link\ref{Lemma-Delta2-H} and \eqref{tensor-curvature-N(c)}. We have: 
\begin{eqnarray}\label{Curv-tensor-expr-1}\nonumber
\sum_{i=1}^m R^{N(c)} \left( \overline{\Delta}\tau(\varphi),d \varphi(X_i)\right ) d \varphi (X_i)&=&c\, m \sum_{i=1}^m  \Big \{ \langle  d \varphi (X_i), d \varphi (X_i) \rangle  \overline{\Delta}{\mathbf H}- \langle  d \varphi (X_i), \overline{\Delta}{\mathbf H}  \rangle d \varphi (X_i) \Big \}\\
&=&c\, m \big \{ m \alpha |A|^2 \eta - 0 \big \}=c\, m^2  \alpha |A|^2 \eta \,.
\end{eqnarray}
Similarly, we compute
\begin{equation}\label{Curv-tensor-expr-2}
\sum_{i=1}^m R^{N(c)} \left( \nabla^{\varphi}_{X_i} \tau(\varphi),\tau(\varphi) \right ) d \varphi (X_i)=c\, m^3   \alpha^3  \eta \,.
\end{equation}
Using \eqref{Delta-2-H-formula}, \eqref{Curv-tensor-expr-1} and \eqref{Curv-tensor-expr-2} into \eqref{2s+1-tension} we obtain the explicit expression of the $3$-tension field:
\begin{equation*}\label{tri-tension-field-explicit}
\tau_3(\varphi)=m \alpha \big[\Delta |A|^2 + |A|^4 -m\, c |A|^2 -m^2\, c\, \alpha^2 \big ] \eta + 2m\alpha A \left ( \grad |A|^2 \right ) 
\end{equation*}
Therefore, we conclude that $M^m$ is a triharmonic hypersurface in $N^{m+1}(c)$ if and only if either it is minimal or
\begin{equation}\label{tri-harmonicity-system-explicit}
\left \{
\begin{array}{ll}
{\rm (i)}\quad &\Delta |A|^2 + |A|^4 -m\, c \,|A|^2 -m^2\, c\, \alpha^2=0 \\
\\
{\rm (ii)}\quad & A \left ( \grad |A|^2\right ) =0 \,.
\end{array}
\right .
\end{equation}
From now on, we assume that $M^m$ is not minimal and we use the hypothesis $c \leq 0$. If $\grad |A|^2=0$ on $M^m$, then $\Delta |A|^2=0$. Therefore, from \eqref{tri-harmonicity-system-explicit}(i) we deduce $|A|^4=0$ which implies that $\alpha =0$, i.e., $M^m$ is minimal (a contradiction). Assume now that there exists a point $p_0 \in M^m$ such that $\grad |A|^2(p_0)\neq 0$. Then  $\grad |A|^2\neq 0$ in a neighbourhood $U$ of $p_0$.  We deduce from \eqref{tri-harmonicity-system-explicit}(ii) that $0$ is an eigenvalue of $A$ on $U$. Now we use the assumption that $m=2$. Since $M^2$ is CMC, it is immediate to conclude that the two principal curvatures are constant on $U$. But then also $|A|^2$ is constant on $U$: this contradiction ends the proof.  
\end{proof}
Now, it is convenient to proceed to the
\begin{proof}[Proof of Theorem\link\ref{Th-structure-surfaces-c>0}] From the proof of Theorem\link\ref{Th-non-existence-surfaces-c<=0} we know that $\grad (|A|^2)= 0$, i.e., $|A|^2$ is a constant on $M^2$. It follows that, since $M^2$ is CMC,
\[
k_1+k_2={\rm constant} \quad {\rm and } \quad k_1^2+k_2^2={\rm constant} 
\]
and consequently $M^2$ is an isoparametric surface in $\s^3$. Then $M^2$ is an open part of either a small hypersphere $\s^2(R)$ of $\s^3$, $0<R\leq 1$, or a Clifford torus $\s^1(R_1) \times \s^1(R_2)$, 
$R_1^2+R_2^2=1$. But, since $M^2$ is proper triharmonic, we know from \cite{Maeta1, Maeta2} that the only possibility is that $M^2$ is an open part of $\s^2(1/\sqrt 3)$.
\end{proof}
\begin{proof}[Proof of Theorem\link\ref{Th-compact}] The Cauchy inequality tells us that, for any hypersurface, we have:
\begin{equation}\label{Cauchy-inequality}
|A|^2=\sum_{i=1}^m k_i^2 \geq \frac{\left (\sum_{i=1}^m  k_i\right )^2}{m}=m f^2 \,.
\end{equation}
Applying \eqref{Cauchy-inequality} to \eqref{tri-harmonicity-system-explicit}(i) (with $c=1$) we immediately deduce
\[
\Delta |A|^2 \leq |A|^2 (2m-|A|^2) 
\]
and, integrating on the compact manifold $M^m$,
\[
0 \leq \int_{M^m}|A|^2 (2m-|A|^2) \,dV\,. 
\]
Thus, either $|A|^2=2m$ or there exists a point $p_1 \in M^m$ such that $|A|^2(p_1)<2m$. On the other hand, \eqref{tri-harmonicity-system-explicit}(i) (with $c=1$) implies 
\[
\Delta |A|^2 > |A|^2 (m-|A|^2) 
\]
on $M^m$. Therefore integration yields
\[
0 > \int_{M^m}|A|^2 (m-|A|^2) \,dV\,. 
\]
Thus there exists a point $p_2 \in M^m$ such that $|A|^2(p_2)>m$. Finally, as $M^m$ is connected, by an obvious continuity argument we deduce that, if $|A|^2=2m$ does not hold on $M^m$, then there exists a point $p \in M^m$ such that $m<|A|^2(p)<2m$.

It only remains to prove that, if $|A|^2=2m$ on $M^m$, then $M^m=\s^m(1/\sqrt 3)$. Indeed, if $|A|^2=2m$ on $M^m$, then \eqref{tri-harmonicity-system-explicit}(i) (with $c=1$) implies $|A|^2=m \alpha^2$. But this is equivalent to say that we have equality in \eqref{Cauchy-inequality}, which means that $M^m$ is umbilical. Then $M^m=\s^m(R)$, $0<R\leq1$ and, as $M^m$ is proper triharmonic, we know that the only possibility is $R=1/\sqrt 3$.
\end{proof}
Theorems\link\ref{Th-non-existence-hypersurfaces-c<=0} and \ref{Th-existence-hypersurfaces-c>0} are an immediate consequence of the following general result:
\begin{theorem}\label{Th-existence-hypersurfaces-c>0-e-c<0}
Let $M^m$ be a non-minimal CMC hypersurface in $N^{m+1}(c)$ and assume that $|A|^2$ is constant. Then
$M^m$ is proper $r$-harmonic ($r \geq 3$) if and only if
\begin{equation*}\label{r-harmonicity-condition-in-general}
|A|^4-m\,c\,|A|^2-(r-2)m^2 \,c\, \alpha^2=0 \,.
\end{equation*}
\end{theorem} 
\begin{proof} As $|A|^2$ is constant, it follows from Lemma\link\ref{Lemma-Delta2-H} that
\begin{equation}\label{Delta-H-A2-constant}
\overline{\Delta}^2 {\mathbf H}= \alpha  |A|^4  \eta\,.
\end{equation}
Now, we show that
\begin{equation}\label{Delta-eta}
\overline{\Delta}\eta=|A|^2 \eta \,.
\end{equation}
Indeed, $\nabla_{X_i}^\varphi \eta=-A(X_i)$ and so
\[
\nabla_{X_i}^\varphi \nabla_{X_i}^\varphi \eta=-(\nabla A)(X_i,X_i)-|A(X_i)|^2\eta\,.
\]
Summing over $i$ and using Lemma~\ref{Lemma-tecnico-1} we obtain
\[
\overline{\Delta}\eta=m \grad f + |A|^2 \eta=|A|^2 \eta \,.
\]
Next, putting together \eqref{Delta-H-formula}, \eqref{Delta-H-A2-constant} and \eqref{Delta-eta}, we easily deduce that
\begin{equation}\label{Delta-H-potenza-p}
\overline{\Delta}^p {\mathbf H}= \alpha  |A|^{2p}  \eta \quad \quad \forall p \in \n^* \,.
\end{equation}
Now we are in a good position to perform the explicit calculation of the $r$-tension field $\tau_{r}(\varphi)$ described in \eqref{2s-tension}, \eqref{2s+1-tension}. We begin with $\tau_{2s}(\varphi)$, $s \geq2$. Using \eqref{Delta-H-potenza-p}, \eqref{tensor-curvature-N(c)} and computing we obtain (as in \eqref{2s-tension}, we omit to write the sum over $i$):
\begin{eqnarray*}
\frac{1}{m} \tau_{2s}(\varphi)&=&\alpha |A|^{4s-2} \eta-c  \,\Big \{ \langle d\varphi(X_i),d\varphi(X_i)   \rangle \alpha |A|^{4s-4} \eta - \langle d\varphi(X_i), \alpha |A|^{4s-4} \eta   \rangle  d\varphi(X_i) \Big \}\\
&& -c\,  m \sum_{\ell=1}^{s-1}  \Big \{ \langle d\varphi(X_i), \alpha |A|^{2s-2\ell-2} \eta   \rangle \big (  -\alpha |A|^{2s+2\ell-4} A(X_i) \big )  \\
&& - \langle d\varphi(X_i), -\alpha |A|^{2s+2\ell-4} A(X_i)   \rangle \alpha |A|^{2s-2\ell-2} \eta \Big \}\\
&&+c\,  m \sum_{\ell=1}^{s-1}  \Big \{ \langle d\varphi(X_i),- \alpha |A|^{2s-2\ell-2} A(X_i)   \rangle  \alpha |A|^{2s+2\ell-4} \eta \\
&& - \langle d\varphi(X_i), \alpha |A|^{2s+2\ell-4} \eta   \rangle  \big (  -\alpha |A|^{2s-2\ell-2} A(X_i) \big )   \Big \}\\
&=&\alpha |A|^{4s-2} \eta-c \, m \alpha |A|^{4s-4} \eta \\
&& -c\,  m \Big \{ \sum_{\ell=1}^{s-1} \big [m \alpha^3 |A|^{4s-6} \eta \big ] + \sum_{\ell=1}^{s-1} \big [m \alpha^3 |A|^{4s-6} \eta \big ] \Big \}\\
&=&\alpha |A|^{4s-6} \Big \{|A|^4-m \,c\, |A|^2-(2s-2)m^2\,c\,\alpha^2 \Big \}\eta\,.
\end{eqnarray*}
This completes the proof in the case $r=2s$. The case $r=2s+1$ is similar and so we omit the details.
\end{proof}
\begin{proof}[Proof of Theorem\link\ref{Th-structure-hypersurfaces-c>0}] (i) Since \eqref{r-harmonicity-condition-in-spheres} holds, we use the Cauchy inequality \eqref{Cauchy-inequality} and deduce:
\[
|A|^4=m|A|^2+(r-2)m^2 \alpha^2 \leq m|A|^2+(r-2)m |A|^2\,,
\]
and, consequently,
\[
|A|^2 \leq m(r-1)
\]
and $|A|^2 = m(r-1)$ if and only if $M^m$ is umbilical, i.e., $M^m$ is an open part of some small hypersphere $\s^m(R)$, $0<R \leq 1$. But, since $M^m$ is proper $r$-harmonic, we know that the only possibility is $R=1/\sqrt r$ (see \cite{BMOR1, Mont-Ratto4}). On the other hand, as $\alpha \neq 0$,  
\[
|A|^4=m|A|^2+(r-2)m^2 \alpha^2 \qquad \Rightarrow \qquad |A|^2>m
\]
and so the proof of statement (i) is complete.

(ii) Again, we start with \eqref{r-harmonicity-condition-in-spheres} and deduce
\[
|A|^2=\frac{m}{2} \Big ( 1+\sqrt{1+4(r-2)\alpha^2}\Big )\,.
\]
But $|A|^2 \geq m \alpha^2$ implies
\begin{equation}\label{*}
1+\sqrt{1+4(r-2)\alpha^2} \geq 2 \alpha^2\,.
\end{equation}
We observe that when $\alpha^2=r-1$ we have equality in \eqref{*}. From this it is easy to conclude that 
\[
1+\sqrt{1+4(r-2)\alpha^2} \geq 2 \alpha^2 \qquad  \Leftrightarrow \qquad \alpha^2 \in (0,r-1] \,.
\] 
Finally, when $\alpha^2=r-1$, we have
\[
|A|^2=\frac{m}{2} \Big( 1+\sqrt{1+4(r-2)(r-1)}\Big )=m(r-1)=m \alpha^2 \,,
\]
so $M^m$ is umbilical and the conclusion is as in part (i).
\end{proof}

\section{Proof of the results on $r$-harmonic isoparametric hypersurfaces}

\begin{proof}[Proof of Theorem\link\ref{Th-isoparametric-p=3}]
We know that $m_1=m_2=m_3$ and using \eqref{principal-curvatures} with $\ell=3$ we can compute $|A|^2=m_1 \left ( \sum_{i=1}^3 k_i^2\right )$ and $\alpha^2=(1/9)\left ( \sum_{i=1}^3 k_i\right )^2$. Then, using standard trigonometric identities, we find that equation \eqref{r-harmonicity-condition-in-spheres} becomes
\begin{eqnarray}\label{explicit-r-harm-p=3}
&&m_1^2 \Big [(2-r) \left(\tan \left(\frac{\pi}{6}-s \right)-\tan \left(s+\frac{\pi }{6}\right)+\cot
   s\right)^2\\\nonumber
   &&+\left(\tan ^2\left(\frac{\pi}{6}-s \right)+\tan ^2\left(s+\frac{\pi
   }{6}\right)+\cot ^2 s\right)^2\\\nonumber
   &&-3 \left(\tan
   ^2\left(\frac{\pi}{6}-s \right)+\tan
   ^2\left(s+\frac{\pi }{6}\right)+\cot ^2 s\right)\Big ]=0\,.
\end{eqnarray}
Next, using the standard addition and subtraction formulas for trigonometric functions, we compute and find that \eqref{explicit-r-harm-p=3} is equivalent to
\begin{equation}\label{eq-key-isoprametric-p=3}
\frac{9 m_1^2 \csc ^4 s \,\big [r \cos (12 s)-r+28 \cos (6 s)+44\big ]}{8 (2 \cos (2 s)+1)^4}=0\,.
\end{equation}
Now, let $x=\cos (6s)$. Then \eqref{eq-key-isoprametric-p=3} is equivalent to
\[
 r x^2+14 x+22-r=0 \,,
\]
whose solutions are
\begin{equation}\label{radici-p=3}
x_1=\frac{-\sqrt{r^2-22 r+49}-7}{r} \quad {\rm and} \quad x_2=\frac{\sqrt{r^2-22 r+49}-7}{r} \,.
\end{equation}
Now, a routine inspection of \eqref{radici-p=3} shows that there is no acceptable solution if $r\leq 19$. By contrast, when $r\geq 20$, we have $-(1/2) \geq x_1 >- 1$ and $-(1/5) \leq x_2 < 1$. From this the conclusion of the proof is immediate.

Finally, we observe that the two hypersurfaces $M_s$, $M_{s'}$ associated to $x_1$ are congruent and their distance from the minimal hypersurface $M_{\pi/6}$ tends to zero as $r$ grows to $+\infty$. Indeed, since $s'=\pi/ \ell -s=\pi/ 3 -s$, it is immediate to deduce that, if we write $M_s=F^{-1}(a)$, where $a= \cos 3 s$, then $M_{s'}=F^{-1}(-a)$. Now, since $F$ is the restriction to $\s^{m+1}$ of a homogeneous polynomial of degree $3$ (odd), then $M_s$ and $M_{s'}$ are congruent via the antipodal map. Similarly, the two hypersurfaces associated to $x_2$ are congruent and each of them approaches one of the focal varieties as $r$ increases to $+\infty$, so the proof is completed.
\end{proof}
\begin{proof}[Proof of Theorem\link\ref{Th-isoparametric-p=6}]
The proof follows exactly the scheme of the proof of Theorem\link\ref{Th-isoparametric-p=3}. Here we have $m_1=m_2=\ldots=m_6$ and instead of \eqref{explicit-r-harm-p=3} we find
\begin{eqnarray}\label{explicit-r-harm-p=6}\nonumber
&&m_1^2 \Big [(2-r) \left(-\tan \left(\frac{\pi}{6}-s \right)+\tan s+\tan \left(s+\frac{\pi
   }{6}\right)+\cot\left(\frac{\pi}{6}-s \right)-\cot s-\cot \left(s+\frac{\pi
   }{6}\right)\right)^2\\\nonumber
   && -6 \left(\tan ^2\left(\frac{\pi}{6}-s \right)+\tan ^2 s+\tan ^2\left(s+\frac{\pi
   }{6}\right)+\cot ^2\left(\frac{1}{6} (\pi -6
   s)\right)+\cot ^2 s+\cot ^2\left(\frac{\pi}{6}-s \right)\right)\\\nonumber
   &&\left(\tan ^2\left(\frac{\pi}{6}-s \right)+\tan
   ^2 s+\tan ^2\left(s+\frac{\pi }{6}\right)+\cot
   ^2\left(\frac{\pi}{6}-s \right)+\cot
   ^2 s+\cot ^2\left(s+\frac{\pi
   }{6}\right)\right)^2 \Big ]=0
\end{eqnarray}
which turns out to be equivalent to
\begin{equation}\label{eq-key-isoprametric-p=6}
\frac{9 m_1^2 \csc ^4 s \sec ^4 s \, \big [r \cos (24 s)-r+64 \cos (12
   s)+224 \big ]}{32 (2 \cos (4 s)+1)^4}=0\,.
\end{equation}
Now, let $x=\cos (12s)$. Then \eqref{eq-key-isoprametric-p=6} is equivalent to
\[
 r x^2+32 x+112-r=0 \,,
\]
whose solutions are
\begin{equation}\label{radici-p=6}
x_1=\frac{-\sqrt{r^2-112 r+256}-16}{r} \quad {\rm and} \quad x_2=\frac{\sqrt{r^2-112 r+256}-16}{r} \,.
\end{equation}
Now, a routine inspection of \eqref{radici-p=6} shows that there is no acceptable solution if $r\leq 109$. By contrast, when $r\geq 110$, we have $-(1/5) \geq x_1 >- 1$ and $-(1/11) \leq x_2 < 1$. Now the conclusion of the proof is as in Theorem\link\ref{Th-isoparametric-p=3}. The only difference, in this case, is that $F$ is the restriction to $\s^{m+1}$ of a homogeneous polynomial of even degree, and so we cannot conclude that the two solutions associated to $x_1$ or $x_2$ are congruent.  In particular, we point out that Miyaoka showed that, when $\ell=6$ and $m_1=m_2=1$, the two focal varieties $F^{-1}(1)$ and $F^{-1}(-1)$ are not congruent (see \cite{Miyaoka-Osaka}).
\end{proof}
\begin{proof}[Proof of Theorem\link\ref{Th-isoparametric-p=4}] Again, we use \eqref{principal-curvatures}. Without making any assumption on $m_1,m_2$ we find that \eqref{r-harmonicity-condition-in-spheres} becomes:
\begin{eqnarray}\label{explicit-r-harm-p=4}
&&-4 (r-2) (m_1 \cot (2 s)-m_2 \tan (2
   s))^2\\ \nonumber
  && +\left(m_1 \tan ^2 s+m_1 \cot
   ^2 s+m_2 \tan ^2\left(s+\frac{\pi
   }{4}\right)+m_2\cot ^2\left(s+\frac{\pi
   }{4}\right)\right)^2
   \\\nonumber
   &&-2 (m_1+m_1)
   \left(m_1 \tan ^2 s +m_1 \cot
   ^2 s+m_2 \tan ^2\left(s+\frac{\pi
   }{4}\right)+m_2 \cot ^2\left(s+\frac{\pi
   }{4}\right)\right)=0\,.
\end{eqnarray}
Now we use the assumption $m_1=m_2$ and the proof follows the patterns of Theorem\link\ref{Th-isoparametric-p=3}. In particular, when $m_1=m_2$ \eqref{explicit-r-harm-p=4} becomes
\begin{equation}\label{eq-key-isoprametric-p=4}
2 m_1^2 \csc ^4(4 s) \, \big [r \cos (16 s)-r+40 \cos (8 s)+88 \big ]=0\,.
\end{equation}
Next, let $x=\cos (8s)$. Then \eqref{eq-key-isoprametric-p=4} is equivalent to
\begin{equation}\label{Pol-semplice-p=4}
r x^2+20 x+44-r=0 \,,
\end{equation}
whose solutions are
\begin{equation}\label{radici-p=4}
x_1=\frac{-\sqrt{r^2-44 r+100}-10}{r} \quad {\rm and} \quad x_2=\frac{\sqrt{r^2-44 r+100}-10}{r} \,.
\end{equation}
As above, inspection of \eqref{radici-p=4} shows that there is no acceptable solution if $r\leq 41$. By contrast, when $r\geq 42$, we have $-(1/3) \geq x_1 >- 1$ and $-(1/7) \leq x_2 < 1$ and so the proof ends as in Theorem\link\ref{Th-isoparametric-p=6}.
\end{proof}
\begin{proof}[Proof of Theorem\link\ref{Th-isoparametric-p=4-m1-diverso-m2}] We simplify \eqref{explicit-r-harm-p=4} without the assumption $m_1=m_2$. We obtain
\begin{eqnarray}\label{eq-key-isoprametric-p=4-m1-diverso-m2}
&&4 r (m_1+m_2)^2+4 m_2 \sec ^2(2 s) \left[2 m_1-m_2 (r+4)+4 m_2\sec ^2(2
   s)\right]\\ \nonumber
   &&+\frac{1}{8}m_1 \csc ^4 s \,\sec ^4 s \,\big[\cos (4 s) (m_1 (r+4)-2
  m_2)-m_1 (r-4)+2 m_2 \big]=0\,.
\end{eqnarray}
Next, we set $y=\cos^2(2s)$. Then a straightforward computation shows that, in terms of $y$, \eqref{eq-key-isoprametric-p=4-m1-diverso-m2} 
becomes, up to the multiplicative quantity $4 m_1^2/ y$, $P_{b,r}(y)=0$, where $P_{b,r}(y)$ is the polynomial defined in \eqref{Pol-grado-4-p=4}. 
This completes the proof of the first part of the theorem. 

Next, we define
\[
T_r(s)= |A|^4(s)-m|A|^2(s)-(r-2)m^2 \alpha^2(s) \,,
\]
where $|A|^2(s)$ and $\alpha^2(s)$ are computed using the principal curvatures given in \eqref{principal-curvatures} with $\ell=4$ and $m=2(m_1+m_2)$.  We observe that, if $r_1>r_2$, then $T_{r_1}(s)\leq T_{r_2}(s)$ on $(0,\pi/4)$. For any fixed value of $r$, we have
\begin{equation}\label{limiti-Tr(s)}
\lim_{s \to 0^+} T_r(s)=\lim_{s \to (\pi/4)^-} T_r(s)=+\infty\,.
\end{equation}
Moreover, for future use we point out that, if we restrict $r$ to any arbitrary set $(0,K]$ ($K>0$), there always exist a sufficiently small $\varepsilon > 0$ such that 
\begin{equation}\label{compactness}
T_r(s) \geq 1 \qquad \forall s \in (0,\varepsilon] \cup \left [\frac{\pi}{4}-\varepsilon,\frac{\pi}{4} \right ) \,\, {\rm and} \,\, \forall r\in (0,K] \,.
\end{equation}
Next, we set ${\mathcal R}^*={\rm Inf} X$, where $X$ is the set defined as follows:
\[
X=\left \{r \in (0,+\infty)  \colon {\rm Min}\left\{ T_r(s),0<s< \pi/4 \right \}<0 \right \}\,.
\]
For future reference, we point out that inspection of \eqref{eq-key-isoprametric-p=4-m1-diverso-m2} shows that ${\mathcal R}^*$ depends on $b=m_1/m_2$.
We can choose $s_0 \in (0, \pi/4)$ such that $\alpha^2(s_0)\neq 0$. It follows that $T_r(s_0)<0$ provided that $r$ is sufficiently large. This fact, together with \eqref{limiti-Tr(s)}, shows that the open set $X$ is not  empty. In fact, $X$ is an open half-line. On the other hand, we know from \cite{Urakawa-no-bihar} that $|A|^4(s)-m|A|^2(s) > 0$ on $(0,\pi/4)$, and from this it is immediate to conclude that ${\mathcal R}^* \geq 2$. 
We note that ${\mathcal R}^* \not\in X$, i.e. $T_{{\mathcal R}^*}(s)\geq 0$ on $(0,\pi/4)$, and $T_r(s)> 0$ on $(0,\pi/4)$ for all $r<{\mathcal R}^*$. 

Next, we shall prove that ${\mathcal R}^*>4$.
To this purpose, first we observe that, if $0<r \leq 4$, $P_{0,r}(y)$ is positive on $(0,1)$. Then it suffices to show that, if $ 0< r \leq 4$, for any $y \in (0,1)$ we have
\begin{equation}\label{Dis-su-P-b-r}
  P_{b,r}(y) >P_{0,r}(y) \quad \quad \forall b>0 \,.
\end{equation}
The property \eqref{Dis-su-P-b-r} is an immediate consequence of the fact that, for all $b \geq 0$,
\[
\frac{dP_{b,r}(y)}{db}=  2 \,(1-y)^3 \,(4-ry)\,b+2\, y\, (1-y)\big( 1+y r (1-y)  \big )>0 
\]
provided that $0< r \leq 4$ and $0< y<1$. Thus $P_{b,r}(y)>0$ for any $b>0$, $y\in(0,1)$ and $0<r \leq 4$, so the equation $T_4(s)=0$ does not admit any solution in $(0,\pi/4)$ and consequently $T_4(s)>0$ on this interval. The last inequality implies $\mathcal{R}^* \geq 4$. 

Then it remains to prove that $\mathcal{R}^* > 4$. We assume $\mathcal{R}^* = 4$ and derive a contradiction. Using \eqref{compactness} with $K=5$ we deduce that, since $\mathcal {R}^* = 4={\rm Inf}X$, for any $n \in \n^*$ there exists $s_n \in (\varepsilon, \pi/4-\varepsilon)$ such that $T_{4+1/n}(s_n)<0$. Now we can extract a subsequence of $\{ s_n \}$ which converges to some $s_1\in [\varepsilon, \pi/4-\varepsilon]$ and passing to the limit we get $T_4(s_1)\leq 0$. But this represents a contradiction, as $T_4(s)$ is positive on $(0,\pi /4)$.
By a similar continuity argument, using again \eqref{compactness} we also deduce that there exists an interior point $\hat{s}$ such that $T_{{\mathcal R}^*}(\hat{s})=0$. 

Finally, if ${\mathcal R}^*$ is an integer we set $r^*={\mathcal R}^*$, otherwise $r^*=\lfloor {\mathcal R}^* \rfloor+1$, where $\lfloor x\rfloor$ denotes the integer part of $x \in \R$. By construction, this definition of $r^*$ ensures that $r^*$ has the properties stated in ${\rm (i),(ii)}$ of the theorem. 

By way of summary, it only remains to prove statement ${\rm (iii)}$. We observe that $\alpha(s)$ is a strictly decreasing function on $(0,\pi /4)$ and $\lim_{s \to 0^+} \alpha(s)=+\infty,\,\lim_{s \to (\pi/4)^-} \alpha(s)=-\infty$. Therefore, there exists a unique value $s^* \in (0,\pi /4)$ such that $\alpha(s^*)=0$ (note that $s^*=(1/2) \arccos \left (\sqrt{m_2/(m_1+m_2)} \right )$ and $M_{s^*}$ is the minimal isoparametric hypersurface of the family). It follows that $T_r(s^*)$ is a positive real number which does not depend on $r$. From this, arguing similarly to the construction of $r^*$ separately on $(0,s^*)$ and $(s^*, \pi/4)$, it is easy to conclude that there exists $r^{**}$ with the property stated in ${\rm (iii)}$ and so the proof is completed. Of course, two solutions are always in the interval $(0,s^*)$, while the other two must belong to $(s^*, \pi/4)$.
\end{proof}
\subsection{Upper estimates for $r^*$ and $r^{**}$}\label{Remark-four-solutions} The arguments of the proof of Theorem\link\ref{Th-isoparametric-p=4-m1-diverso-m2} apply to all the families of isoparametric hypersurfaces. In particular, in the situations of Theorems\link\ref{Th-isoparametric-p=3}, \ref{Th-isoparametric-p=6}, \ref{Th-isoparametric-p=4} we were able to determine exactly $r^*=r^{**}=20,110,42$ respectively. By contrast, in the case of isoparametric hypersurfaces of degree $4$ and different multiplicities, it seems to be a difficult task to determine the explicit values of $r^*=r^*(b)$ and $r^{**}=r^{**}(b)$, $b=(m_2/m_1)>1$. The aim of this subsection is to explain the type of difficulties which arise and also provide some upper bounds for both $r^*$ and $r^{**}$. 

We rewrite the polynomial $P_{b,r}(y)$ in \eqref{Pol-grado-4-p=4} as follows:
\[
P_{b,r}(y)=Q_b(y)-r R_b(y)\,,
\]
where
\[
\begin{array}{l}
Q_b(y)=2 \big [2 b^2 (1-y)^3+b \,y \,(1-y)+2 y^3 \big ] \\
R_b(y)=(1-y)\, y\, [b (y-1)+y]^2 \,.
\end{array}
\]
We note that $Q_b(y)$ is positive on $(0,1)$, while $R_b(y)$ is positive on $(0,y_0)$ and $(y_0,1)$, where $y_0=b/(1+b)=m_2/(m_1+m_2)$ corresponds to the minimal hypersurface $M_{s_0}$, where
\[
s_0=\frac{1}{2}\arccos \left ( \sqrt{y_0} \right ) \,.
\]
Moreover, $R_b(y_0)=0$ and $Q_b(y_0)=P_{b,r}(y_0)= 6b^2/(1+b)^2 >0$.
Now, since we want to determine the smallest value of $r$ such that $P_{b,r}(y)$ admits a root in $(0,1)$, it is natural to define
\[
{\mathcal R}(y)= \frac{Q_b(y)}{R_b(y)} \quad \quad y \in (0,1)\,,\,\, y \neq y_0\,,
\]
and set
\[
\begin{array}{l}
{\mathcal R}_1={\rm Inf} \Big \{{\mathcal R}(y) \colon y \in (0,y_0) \Big \} \\
{\mathcal R}_2={\rm Inf} \Big \{{\mathcal R}(y) \colon y \in (y_0,1) \Big \}\,.
\end{array}
\]
The connection with the proof of Theorem\link\ref{Th-isoparametric-p=4-m1-diverso-m2} is the following:
\[
{\mathcal R}^*= {\rm Min} \big \{{\mathcal R}_1, {\mathcal R}_2 \big \} \,\,; \quad r^{**}=\lfloor {\mathcal R}^{**} \rfloor +1,\quad {\rm where}\,\, {\mathcal R}^{**}= {\rm Max} \big \{{\mathcal R}_1, {\mathcal R}_2 \big \} \,.
\] 
The function ${\mathcal R}(y)$ is positive and tends to $+\infty$ when $ y \rightarrow 0^+$, $ y \rightarrow y_0^{\pm}$ and  $ y \rightarrow 1^-$. Therefore ${\mathcal R}(y)$ admits a minimum point $y_1$ on $(0,y_0)$ and a minimum point $y_2$ on $(y_0,1)$, so that ${\mathcal R}_i={\mathcal R}(y_i)$, $i=1,2$. Unfortunately, it is difficult to compute the exact values of $y_1$ and $y_2$. A numerical analysis shows that, when $b$ is not too big, $y_1$ is on the interval $(0,y_0/2)$, while $y_1$ belongs to $(y_0/2,y_0)$ for large values of $b$. Similarly, $y_2$ increases as $b$ does. Therefore, a first reasonable upper estimate for ${\mathcal R}_1$ can be obtained by considering the medium point. We have 
\[
{\mathcal R}_1 \leq {\mathcal R}(y_0/2)=\frac{8 \left(b^2+6 b+10\right)}{b (b+2)} \,.
\] 
Similarly, an upper estimate for ${\mathcal R}_2$ is given by 
\[
{\mathcal R}_2 \leq {\mathcal R}\big ((1+y_0)/2 \big )=\frac{(8 + 48 b + 80 b^2)}{(1 + 2 b)}\,.
\]  
Next, we observe that 
\[
{\mathcal R}\big ((1+y_0)/2 \big ) > {\mathcal R}(y_0/2) \quad \forall b > 1 \,.
\]
The conclusion of this analysis is the following.
Let $r^*=r^*(b)$, $r^{**}=r^{**}(b)$ ($b>1$) be the integers defined in Theorem\link\ref{Th-isoparametric-p=4-m1-diverso-m2}. Then
\begin{equation}\label{eq-stime-r*-r**}
\begin{array}{ll}
{\rm (i)} & r^*(b) \leq 1+ \displaystyle{\frac{8 \left(b^2+6 b+10\right)}{b (b+2)} }\\
&\\
{\rm (ii)} & r^{**}(b) \leq 1+\displaystyle{\frac{(8 + 48 b + 80 b^2)}{(1 + 2 b)}}\,.
\end{array}
\end{equation}
For instance, when $ b$ is close to $1$, \eqref{eq-stime-r*-r**} tells us that $r^*(b) \leq 46$ and $r^{**}(b) \leq 46$ (we know from Theorem\link\ref{Th-isoparametric-p=4} that $r^*=r^{**}=42$ when $b=1$). When $b=10000$, \eqref{eq-stime-r*-r**} gives $r^* \leq 9$ and $r^{**}\leq 400005$ (we pointed out in Example\link\ref{Ex-r=5} that, when $b=10000$, $r^*=5$ and $r^{**}=312919$). Another geometrically significant case which occur is $(m_1,m_2)=(7,8)$. In this case, \eqref{eq-stime-r*-r**} with $b=8/7$ yields $r^* \leq 41$ and $r^{**} \leq 51$. A numerical analysis shows that, when $b=8/7$, the exact values are $r^*=38$ and $r^{**}=47$.
\subsection{Remarks} (i) By a suitable choice of the orientation we assumed, for convenience, $m_1 \leq m_2$ so that $b=m_2 / m_1 \geq 1$. We point out that, if one decides to carry out the analysis for all $b>0$, then there are no significant differences because of the  following fact, which is true for all $b\neq 0$, $y, r \in \R$ and can be verified by a direct computation:
\begin{equation*}\label{eq-symm-pol}
 b^2\, P_{1/b,r}(1-y)=P_{b,r}(y)\,.
\end{equation*}
(ii) In order to recover the correspondance between $P_{1,r}(y)$ and the second order equation \eqref{Pol-semplice-p=4}, one has to remember that $y=\cos^2(2s)$, while $x=\cos(8s)$ and consequently perform the changes of variable
\[
y=\frac{1}{4} \left(2\pm\sqrt{2} \sqrt{x+1}\right)\,.
\]
Unfortunately, these changes of variable do not yield any simplification if $b\neq1$.

(iii) We think that it could be interesting to find an isoparametric family with one or three distinct proper solutions for some $r$ (necessarily, the degree of such a family must be $4$ and $m_1 \neq m_2$). A necessary condition to have just one solution is that the real number ${\mathcal R}^*$ defined in the proof of Theorem\link\ref{Th-isoparametric-p=4-m1-diverso-m2} be an integer.

(iv) Using \eqref{r-harmonicity-condition-in-spheres} and \eqref{principal-curvatures} with $\ell=1,2$, together with the technique of proof of the present paper, it is possible to recover the characterisation of proper $r$-harmonic small hyperspheres and generalised Clifford tori obtained by variational methods in \cite{Mont-Ratto4}.

\section{$ES-r$-harmonicity and $r$-harmonicity}

\begin{proof}[Proof of Theorem\link\ref{Th-Es-r-e-r-harmonicity-general}] It is convenient to separate two cases: $r$ even and $r$ odd. First, let us assume that $r=2s \geq 4$. In this case, since if $p\neq p'$ the subspace of $p$-forms is orthogonal to that of $p'$-forms, we have
\begin{equation}\label{E-2s-struttura}
E_{2s}^{ES}(\varphi)= \frac{1}{2} \int_M  \sum_{k=0}^{s-1} \langle \omega_{2k}, \omega_{2k} \rangle  \,dV\,,
\end{equation}
where $\omega_{2k}$ are  $\varphi^{-1}TN$-valued $2k$-differential forms on $M$. Each $\omega_{2k}$ is the sum of $2k$-differential forms obtained by applying to $\tau$ a string of length $2s-2$ made of $d$'s and $d^*$'s. Since $d^*$ is the null operator on $0$-forms, if we read from right to left the string of a nonzero form, the number of $d$'s that we encounter in any truncated string is always not smaller than  the number of $d^*$s, otherwise the form automatically vanishes.  In each of these strings the number $n_d$ of occurrences of $d$ minus the number $n_{d^*}$ of occurrences of $d^*$ is equal to $2k$. To help the reader, let us give an explicit example of a $2k$-form which can occur when $r=26, 2k=8$:
\begin{equation}\label{Stringa-esempio}
\alpha_{8}=d d^*  d^6 (d^*)^3 d^5 (d^* d)^4  \tau  \quad \quad (n_d=16,\,n_{d^*}=8 )\,.
\end{equation}
\begin{definition}\label{Def-omega-good} Let $\alpha_{2k}$ be any $2k$-form obtained applying to $\tau$ any string made of $d$'s and $d^{\ast}$'s. We say that $\alpha_{2k}$ is \textit{good} if its string starts on the right with a term of the type $  d^2 \overline{\Delta}^i \tau$, for some $i$ such that $r-4 \geq i \geq 0$.
\end{definition}
For instance, the $8$-form in \eqref{Stringa-esempio} is good with $i=4$. The following lemma is important for our purposes.
\begin{lemma}\label{Lem-good-forms} Let $\alpha_{2k}$ be a good form in the sense of Definition\link\ref{Def-omega-good}. Then, if the map $\varphi$ satisfies the hypotheses of Theorem\link\ref{Th-Es-r-e-r-harmonicity-general}, $\alpha_{2k}$ vanishes.
\end{lemma}
\begin{proof}[Proof of Lemma\link\ref{Lem-good-forms}] As in \eqref{d2-description}, we have
\begin{equation*}\label{d2-description-lemma}
d^2(\overline{\Delta}^i \tau)(X,Y)=R^{N(c)} \big(d\varphi(X),d\varphi(Y)\big )\overline{\Delta}^i \tau 
\end{equation*}
and so, using \eqref{tensor-curvature-N(c)}, the conclusion of the lemma follows immediately from \eqref{Ipotesi-Deltatau-ortogonale}.
\end{proof}
Next, for $ k \geq 1$,  let $\omega_{2k}$ be the $2k$-form which appears in \eqref{E-2s-struttura}. Then $\omega_{2k}$ is the sum of good forms. Indeed, this is a simple consequence of the fact that $n_d-n_{d^*}=2k \geq 2$. Therefore, since $\langle\, , \rangle $ is bilinear, $\langle \omega_{2k}, \omega_{2k} \rangle $ is the sum of addends of the type $\langle \alpha_{2k}, \beta_{2k} \rangle $, where $\alpha_{2k} ,\beta_{2k} $ are good forms. Now, as in \cite{EL83}, we consider a variation $\{\varphi_t\}_t$ as a map $\Phi: \R \times M \to N$ and denote by $\nabla^\Phi$ the corresponding covariant derivative. Then we can study the first variation and we have:
\begin{eqnarray}\label{First-var-1}\nonumber
\left . \frac{d}{dt} \int_M \langle (\alpha_{2k})_t, (\beta_{2k})_t \rangle dV \right |_{t=0}&=&\int_M \Big \{ \langle \left. \nabla^{\Phi}_{\frac{\partial}{\partial t}} (\alpha_{2k})_t \right |_{t=0}, (\beta_{2k})_0 \rangle +\langle  (\alpha_{2k})_0 ,\left . \nabla^{\Phi}_{\frac{\partial}{\partial t}} (\beta_{2k})_t \right |_{t=0}\rangle \Big \}dV \\
&=&0 \,,
\end{eqnarray}
where the $\varphi_t^{-1}TN$-valued $2k$-form $(\alpha_{2k})_t$ is identified with a section in $\Lambda ^{2k}T^*M \otimes \Phi^{-1}TN$, and the last equality is a consequence of the fact that $(\alpha_{2k})_0$ and $(\beta_{2k})_0$ are good forms and so they vanish by Lemma~\ref{Lem-good-forms}. 

By way of summary, if $\varphi$ is a map which verifies the hypotheses of Theorem\link\ref{Th-Es-r-e-r-harmonicity-general}, the first variation associated to the forms $\omega_{2k}$ ($k \geq 1$) in \eqref{E-2s-struttura} automatically vanishes. It remains to understand the contribution of $\omega_0$. To this purpose, we observe that 
\begin{equation*}\label{Structure-omega-zero}
\omega_0= \overline{\Delta}^{s-1}\tau +{\rm good \,\,forms \,\,of \,\,degree\,\, zero }\,. 
\end{equation*}
It follows that 
\[
\langle \omega_0 , \omega_0 \rangle = \langle\overline{\Delta}^{s-1}\tau ,\overline{\Delta}^{s-1} \tau\rangle+ {\rm addends \,\, of \,\, type \,1}+{\rm addends \,\, of \,\, type \,2}\,,
\]
where the terms of type\link$1,2$ are, respectively, of the form $ \langle \alpha_0,\beta_0 \rangle$ and $ \langle\overline{\Delta}^{s-1}\tau ,\alpha_0\rangle$, where $\alpha_0, \beta_0$ are good forms of degree $0$. Arguing as in \eqref{First-var-1}, we conclude that the first variation of the terms of type $1$ vanishes automatically. As for type $2$, let us first assume that $M$ is \textit{compact}. Then we can use the fact that $d^*$ is the adjoint operator of $d$, i.e.,
\begin{equation}\label{adjoint}
\int_M \langle d \alpha, \beta \rangle dV =\int_M \langle  \alpha,d^* \beta \rangle dV
\end{equation}
whenever $\alpha, \beta$ are, respectively, a $p$-form and a $(p+1)$-form. Therefore, before taking the derivative with respect to $t$, we can transform, by means of suitable iterated applications of \eqref{adjoint}, any term of type $2$ into a term of the type $\langle \alpha_{2k},\beta_{2k} \rangle$, where $ \alpha_{2k},\beta_{2k}$ are good forms ($k \geq1$) (it is easy to check that, in this process, the maximum value which we can obtain for the exponent $i$ of Definition\link\ref{Def-omega-good} is equal to $r-4$). 

The case that $M$ is \textit{not} compact can be handled by using \textit{compactly supported variations}. More precisely, we consider all compact subsets $D \subset M$ with smooth boundary and define
\[
E_r^{ES}(\varphi;D)= \frac{1}{2}\int_D |(d^*+d)^r(\varphi)|^2 dV \,.
\]
For each subset $D$ we consider all smooth variations of $\varphi$ which have their support in $D$, i.e., $\Phi=\{ \varphi_t\}_t $ such that $\varphi_t=\varphi$ on $M \backslash D$ for any $t$. Denoting by $V$ its variation vector field, we observe that obviously $V$ and its covariant derivatives vanish on $M \backslash D$ and, by continuity, on  
$\overline{M \backslash D}$. Now, the only difference with respect to the compact case appears when we study the first variation of the terms of type $2$. However, we can again proceed as above because \eqref{adjoint} remains true if one integrates on $D$ instead of $M$. Indeed, as $\nabla^{\Phi}_{\partial /\partial t} ( \cdot )|_{t=0}$ and $d$, or $d^*$, commute (on the whole $M$), we have
\begin{eqnarray*}
\left . \frac{d}{dt} \int_D \langle d\alpha_t,\beta_t\rangle dV \right |_{t=0}&=&\int_D \Big \{\langle \left . \nabla^{\Phi}_{\frac{\partial}{\partial t}}(d \alpha_t) \right |_{t=0}, \beta_0 \rangle +\langle d\alpha_0,\left . \nabla^{\Phi}_{\frac{\partial}{\partial t}} \beta_t \right |_{t=0}\rangle \Big \} dV\\
&=&\int_D \Big \{\langle d \left (  \left . \nabla^{\Phi}_{\frac{\partial}{\partial t}} \alpha_t \right |_{t=0} \right ), \beta_0 \rangle +\langle d\alpha_0,\left . \nabla^{\Phi}_{\frac{\partial}{\partial t}} \beta_t \right |_{t=0}\rangle \Big \} dV\,.
\end{eqnarray*}
We note that the terms like $ \nabla^{\Phi}_{\partial /\partial t}  \alpha_t \big |_{t=0}$ can always be expressed as functions of $V$ and its covariant derivatives. For example,
\[
\left . \nabla^{\Phi}_{\frac{\partial}{\partial t}} \tau(\varphi_t) \right |_{t=0}=-\overline{\Delta}V-\trace R^N\big ( d\varphi(\cdot),V\big )d\varphi(\cdot)\,.
\]
Then it is well-known that, in general, the difference
\[
\langle d \sigma, \rho \rangle- \langle \sigma, d^* \rho \rangle
\]
is the divergence of a vector field defined on $M$. Since the components of this vector field are given in terms of $V$ and its covariant derivatives, its divergence vanishes on the boundary $\partial D$. Therefore
\begin{eqnarray*}
\left . \frac{d}{dt} \int_D \langle d\alpha_t,\beta_t\rangle dV \right |_{t=0}&=&\int_D \Big \{\langle \left . \nabla^{\Phi}_{\frac{\partial}{\partial t}} \alpha_t \right |_{t=0}, d^* \beta_0 \rangle +\langle \alpha_0,d^* \left (\left . \nabla^{\Phi}_{\frac{\partial}{\partial t}} \beta_t \right |_{t=0}\right )\rangle \Big \} dV\\
&=&\left . \frac{d}{dt} \int_D \langle \alpha_t,d^*\beta_t\rangle dV \right |_{t=0}
\end{eqnarray*}
and now the proof continues as in the compact case.

In summary, under the hypotheses of Theorem\link\ref{Th-Es-r-e-r-harmonicity-general} the only relevant term for the computation of the first variation of the functionals $E_{2s}^{ES}(\varphi)$ is $\langle\overline{\Delta}^{s-1}\tau ,\overline{\Delta}^{s-1} \tau\rangle=|\overline{\Delta}^{s-1}\tau |^2$, i.e.,
\begin{eqnarray*}
\left . \frac{d}{dt} E_{2s}^{ES}(\varphi_t) \right |_{t=0}&=&\frac{1}{2}
\left . \frac{d}{dt} \int_M  \sum_{k=0}^{s-1} \langle (\omega_{2k})_t, (\omega_{2k})_t \rangle  dV \right |_{t=0}\\
&=&\frac{1}{2} \left . \frac{d}{dt} \int_M  \langle (\omega_{0})_t, (\omega_{0})_t \rangle  dV \right |_{t=0}=
\frac{1}{2} \left . \frac{d}{dt} \int_M  \langle\overline{\Delta}^{s-1}\tau(\varphi_t) ,\overline{\Delta}^{s-1} \tau(\varphi_t) \rangle  dV \right |_{t=0}\\
&=& \left . \frac{d}{dt} E_{2s}(\varphi_t) \right |_{t=0}\,.
\end{eqnarray*}
So the critical points of $E_{2s}^{ES}(\varphi)$ coincide with those of $E_{2s}(\varphi)$ and the proof of the case $r=2s$ is ended. 

In the case that $r=2s+1\geq 5$ we have
\[
E_{2s+1}^{ES}(\varphi)= \frac{1}{2} \int_M  \sum_{k=0}^{s-1} \langle \omega_{2k+1}, \omega_{2k+1} \rangle  \,dV\,.
\]
Now the proof is very similar to that of the case $r=2s$ and so we omit  further details.
\end{proof}
\begin{proof}[Proof of Corollary\link\ref{Cor-Es-r-e-r-harmonicity}] Since $\tau=m \mathbf H$, we know from \eqref{Delta-H-potenza-p} that
\[
\overline{\Delta}^i \tau =\big ( m \, \alpha \,|A|^{2 i} \big )\eta  \quad \forall i \geq 0 \,.
\]
Thus \eqref{Ipotesi-Deltatau-ortogonale} holds and so the conclusion follows immediately from Theorem\link\ref{Th-Es-r-e-r-harmonicity-general}.
\end{proof}

\section{Further studies} 

{\rm (i)} We point out that, as in the case of Theorem\link\ref{Th-existence-hypersurfaces-c>0}, the conclusion of Theorem\link\ref{Th-existence-hypersurfaces-c>0-e-c<0} also holds if, in its statement, the word $r$-harmonic is replaced by $ES-r$-harmonic.\\

{\rm (ii)} In the proof of Theorem\link\ref{Th-Es-r-e-r-harmonicity-general}, the geometric assumption \eqref{Ipotesi-Deltatau-ortogonale} has been used, together with the explicit form of the curvature tensor field of a space form, to deduce that
\begin{equation}\label{Ipotesi-forte}
d^2 \overline{\Delta}^i \tau (\varphi)=0 \,, \quad i=1,\ldots, r-4\,.
\end{equation}
Now, let $\varphi:(M,g) \to (N,h)$ be any smooth map between two Riemannian manifolds. We point out that, as a consequence of the proof of Theorem\link\ref{Th-Es-r-e-r-harmonicity-general}, if $\varphi$ satisfies \eqref{Ipotesi-forte}, then $\varphi$ is $r$-harmonic if and only if it is $ES-r$-harmonic.\\

{\rm (iii)} In this paper we have worked with the hypothesis that the mean curvature function $f$ and $|A|^2$ are both constant. In general, if $M^m$ is any hypersurface in a space form of curvature $c$, then these two quantities are related to the scalar curvature $S$ of $M^m$ by the following formula:
\[
S=c(m-1)m+m^2f^2-|A|^2\,.
\]
It follows that if two amongst $f$, $|A|^2$ and $S$ are constant, then so is the third. In the literature, the study of hypersurfaces in $\s^{m+1}$ with constant scalar and mean curvatures has attracted the attention of several geometers. For instance, Chang proved the following result:
\begin{theorem}\label{Th-Chang}\cite{Chang} A compact hypersurface $M^m$ of constant scalar curvature and constant mean curvature in $S^{m+1}$ is isoparametric provided that it has $3$ distinct principal curvatures everywhere.
\end{theorem} 
\noindent To our knowledge, in the literature there are no examples of hypersurfaces $M^m$ in $\s^{m+1}$ which have constant $f$ and $S$ but are \textit{not} isoparametric. Moreover, we know that, when $m=3$, any compact CMC hypersurface  of constant scalar curvature in $\s^4$ is isoparametric (see \cite{Brito,CW}). Therefore, putting together the results of \cite{BMOR1, Maeta2, Mont-Ratto4}, Theorem\link\ref{Th-isoparametric-p=3} and Corollary\link\ref{Cor-Es-r-e-r-harmonicity}, we have a \textit{complete classification} of the compact, $r$-harmonic ($ES-r$-harmonic) hypersurfaces in $\s^4$ with constant $f$ and $|A|^2$. These hypersurfaces are precisely:
\begin{itemize}
\item The small hyperspheres  $\s^3(1/\sqrt r)$;
\item The generalised Clifford tori $\s^1(R_1)\times\s^2(R_2)$ obtained in Theorem\link 1.2 of \cite{Mont-Ratto4}; 
\item The isoparametric hypersurfaces obtained in Theorem\link\ref{Th-isoparametric-p=3} with $m_1=1$.
\end{itemize}\vspace{2mm}

{\rm (iv)} In conclusion, the known results and conjectures concerning biharmonic hypersurfaces in a sphere (see \cite{Balmus, SMCO}) suggest to us to formulate the following:

\vspace{2mm}
\noindent \textbf{Conjecture 1:} \textit{Any proper $r$-harmonic ($ES-r$-harmonic) hypersurface in $\s^{m+1}$ is CMC and has $|A|^2$ equal to a constant.}
\vspace{2mm}

\noindent Or, even stronger:

\vspace{2mm}
\noindent \textbf{Conjecture 2:} \textit{Any proper $r$-harmonic ($ES-r$-harmonic) hypersurface in $\s^{m+1}$ is isoparametric.}\vspace{2mm}

As an additional comment to the above conjectures, we point out that results similar to Theorem\link\ref{Th-Chang} hold without the assumption $M$ compact or complete. Indeed, it is not difficult to show that if a CMC  hypersurface in a space form $N^{m+1}(c)$ has $|A|^2$ constant and at most two distinct principal curvatures at each point, then it is isoparametric of degree either $1$ or $2$. When the hypersurface has three distinct principal curvatures we can  prove the following local version of Theorem~\ref{Th-Chang}:
\begin{proposition} Let $M^m$ be a CMC hypersurface in $\s^{m+1}$ with $|A|^2$ constant. Assume that, at each point, $M$ has three distinct principal curvatures, each of multiplicity at least two. Then $M$ is isoparametric.
\end{proposition}
\begin{proof} As $M$ has a fixed number of distinct principal curvatures at each point, their multiplicities are constant on $M$ and the $k_i$'s are smooth functions. Moreover, the corresponding distributions $T_i$ are integrable and, since $\dim T_i \geq 2$, $X_i k_i=0$ for any vector $X_i \in T_i$, $i=1,2,3$ (see \cite{Otsuki}). Then on $M$ we have:  
\[
\left \{
\begin{array}{l}
m_1 k_1+m_2 k_2+m_3 k_3= m \,\alpha \\
m_1 k_1^2+m_2 k_2^2+m_3 k_3^2= a^2 \,,
\end{array}
\right .
\]

where $a=|A|>0$. Taking the derivative with respect to any $X_1\in T_1$ we get
\[
\left \{
\begin{array}{l}
m_2\, (X_1 k_2)+m_3\, (X_1k_3)=0\\
m_2 k_2\, (X_1 k_2)+m_3 k_3\,( X_1 k_3)= 0 \,.
\end{array}
\right .
\]
Now, since 
\[
\left | 
\begin{array}{cc}
m_2 & m_3 \\
m_2 k_2 & m_3 k_3 
\end{array}
\right | = m_2 m_3 (k_3-k_2) \neq 0 \,,
\]
we deduce that $X_1 k_2=X_1 k_3=0$. In the same way, we obtain $X_2 k_1=X_2 k_3=0$ and $X_3 k_1=X_3 k_2=0$. Therefore, $k_1, k_2$ and $k_3$ are constant functions on $M$ and so $M$ is isoparametric.
\end{proof}

\end{document}